\documentclass[11pt]{article}
\usepackage[top=2.8cm, bottom=2.8cm, left=2.8cm, right=2.8cm]{geometry}
\usepackage{fancyhdr}
\renewcommand{\geq}{\geqslant}
\renewcommand{\leq}{\leqslant}
\renewcommand{\emptyset}{\varnothing}

\renewenvironment{abstract}
{\noindent\textbf{Abstract:}\ignorespaces} 
{\par} 
\usepackage{graphicx} 
\usepackage{subcaption}   
\usepackage{amsfonts}
\usepackage{amsmath,verbatim}
\allowdisplaybreaks
\usepackage{cite}
\usepackage{amsthm}
\newtheorem{theorem}{Theorem}[section]
\newtheorem{definition}[theorem]{Definition}
\newtheorem{lemma}[theorem]{Lemma}
\newtheorem{corollary}[theorem]{Corollary}
\newtheorem{proposition}[theorem]{Proposition}

\numberwithin{equation}{section}
\theoremstyle{remark}
\usepackage{listings}
\usepackage{hyperref}
\hypersetup{
	colorlinks = true,
	citecolor  = red,
	linkcolor  = red,
	urlcolor   = magenta
}
\usepackage{xcolor} 
\usepackage{amssymb}
\usepackage{titlesec}
\titleformat{\section}[block]{\centering}{\thesection}{1em}{}
\titleformat{\subsection}[block] 
{\normalfont\bfseries} 
{\thesubsection} 
{0.5em} 
{} 
\titlespacing*{\subsection}{0pt}{3ex}{0.5em} 
\makeatletter
\renewcommand{\@maketitle}{%
	\begin{center}
		{\large\bfseries \@title \par} 
		\vspace{1em} 
		{\small \@author \par} 
		\vspace{0.2em} 
		{\footnotesize \@date \par} 
	\end{center}
}
\makeatother
\begin{document}
	\thispagestyle{empty}  
	
	\begin{center}
		{\Large\bfseries Dimensions of orbital sets in complex dynamics \par}
		\vspace{1em}
		{\normalsize Jonathan M. Fraser \& Yunlong Xu\footnote{JMF was  financially supported by   an \emph{EPSRC Open Fellowship} (EP/Z533440/1) and a \emph{Leverhulme Trust Research Project Grant} (RPG-2023-281). YX was financially supported by the Chinese Scholarship Council.} \par}
		\vspace{0.2em}
		{\normalsize University of St Andrews \par}
		\vspace{0.2em}
		{\normalsize \texttt{jmf32@st-andrews.ac.uk \& yx71@st-andrews.ac.uk} \par}
	\end{center}
	
	\vspace{1.5em}
	\begin{abstract}
		~Let $ E $ be a non-empty compact subset of the Riemann sphere and $T$ be a rational map   of degree at least two. We study the associated \emph{orbital set}, that is, the backwards orbit of $E$ under $T$, and study the relationship between the upper box dimension of the   orbital set and the upper box dimensions of the   Julia set of $T$ and   the initial set $ E$.  Our results  extend previous work on inhomogeneous iterated function systems to the setting of complex dynamical systems.
	\end{abstract}
	
	\vspace{1em}
	\tableofcontents
	\section{\textbf{Introduction}}
	
	\subsection{Inhomogeneous attractors, Kleinian orbital sets, and dimension}
	
	Many fractals, including self-similar sets and self-affine sets, are generated by an iterated function system (IFS). The theory of IFSs provides a powerful framework for describing the geometry of fractals and for generating complex and representative fractal examples. In 1985, Barnsley and Demko\cite{ref13} generalised the  IFS model to allow for inhomogeneous attractors, which are generated by taking the closure of the orbit of a fixed compact condensation set under a given standard IFS.  This led to much further work on the topic, with a key problem being to relate geometric properties of the inhomogeneous attractor with the IFS and condensation set, see \cite{olsen, ref20} and the references therein for further examples and more details.  
	
We briefly recall the inhomogeneous IFS model. 	Consider an IFS $ \mathcal{S}=\left\lbrace S_{i} \right\rbrace_{i=1}^{n}  ,$ where each function $ S_{i} $ is a contraction and maps a fixed compact subset $X$ of Euclidean space into itself.  Further, fix a compact condensation set $ C \subseteq X $.  Then there exists a unique non-empty compact \emph{inhomogeneous attractor} $ F_{C} $ such that
	$$F_{C}= \bigcup_{S\in\mathcal{S}}S( F_{C}) \cup C .$$
	When $ C $ is empty, the attractor $ F_{\emptyset} $ is the usual (homogeneous) attractor of the IFS. One can also define the \textit{orbital set}, $ O $, by
	$$O=C\cup \bigcup_{k\in\mathbb{N}}\bigcup_{i_{1},\dots,i_{k}\in\left\lbrace 1,\dots,N\right\rbrace }S_{i_{1}}\circ\dots\circ S_{i_{k}}( C), $$
	and observe that
	$$F_{C}=F_{\emptyset}\cup O=\overline{O},$$ 
	see \cite{olsen}. 	Due to the fractal nature of these attractors, it is natural to study their dimensions and we recall the definitions we  use here, referring the reader to \cite{ref2} for more background.
	\begin{definition}
		Let $ A $ be a non-empty bounded subset of   Euclidean space.
		\\
		(a) For   $ \delta>0 $, let $ N_{\delta}\left( A \right)  $ be the minimal number of open sets of diameter at most $ \delta $ required to cover  $ A. $ The \textit{upper box dimension} and the \textit{lower box dimension} of $ A $ are defined by
		$$\overline{\dim}_{\mathrm{B}}A= \limsup_{\delta\rightarrow 0}\dfrac{\log N_{\delta}( A ) }{-\log \delta}$$
		and
		$$\underline{\dim}_{\mathrm{B}}A= \liminf_{\delta\rightarrow 0}\dfrac{\log N_{\delta}( A ) }{-\log \delta} $$respectively.
		When $ \overline{\dim}_{\mathrm{B}}A $ and $\underline{\dim}_{\mathrm{B}}A  $ coincide, we call the common value the \textit{box dimension} of $ A $ and denote it by $ \dim_{\mathrm{B}}A .$
		\\
		(b) The Hausdorff dimension of A is defined by
		$$\dim_{\mathrm{H}}A= \inf\left\lbrace s \geq 0:
		\forall \varepsilon>0, ~\exists \left\lbrace U_{i} \right\rbrace_{i=1}^{\infty} \text{ such that }  A\subseteq \bigcup_{i=1}^{\infty}U_{i} \text{ and } \sum_{i=0}^{\infty}|U_{i}|^{s} <\varepsilon\right\rbrace . $$
	\end{definition}

	Fraser\cite[Corollary 2.2]{ref6} showed that if an IFS consisting of similarity maps satisfies the   open set condition (OSC), then 
	$$\overline{\dim}_{\mathrm{B}}F_{C} =\max\left\lbrace\overline{\dim}_{\mathrm{B}}C, ~\overline{\dim}_{\mathrm{B}}F_{\emptyset} \right\rbrace  .$$ 
	See also \cite{ref18, kaenmaki1, kaenmaki2,olsen} for further results in this direction. 	The formula above may be thought of as the  \textit{expected formula}. Indeed, it is very straightforward to see that	$$\dim_{\mathrm{H}}F_{C} =\max\left\lbrace\dim_{\mathrm{H}}C, ~\dim_{\mathrm{H}}F_{\emptyset} \right\rbrace  $$
	holds without assuming any separation conditions and without any assumptions on the maps in the IFS.    However, in \cite{ref18}, Baker, Fraser, and Máthé demonstrated that the  expected formula does \emph{not} hold in general for upper box dimension when overlaps occur, that is,  when the OSC is not assumed.  Further, the expected formula can fail for \emph{lower} box dimension even for IFSs consisting of similarity maps and satisfying  the OSC \cite{ref6, rutar}. For inhomogeneous self-\emph{affine} sets,  Fraser observed that the expected formula for the upper box dimension may fail, even without overlaps \cite{ref19}. Later, Burrell and Fraser \cite{ref17} showed that  the expected formula holds for self-affine sets in an appropriate generic sense.

It is natural to consider inhomogeneous attractors and orbital sets in more general dynamical settings.  To this end, in 2023 Bartlett and Fraser \cite{ref5} considered orbital sets generated by Kleinian group actions.  These orbital sets are reminiscent of the art of M.C.~Escher who produced intricate pictures generated by the orbit of simple images of lizards and fish etc under a Fuchsian group.  In a certain sense this makes Escher's work some of  the original orbital sets, and highlights the relevance of this model in general dynamical settings. 

We briefly recall the results from \cite{ref5}.  Let $ E $ be a non-empty subset of the Poincar\'{e} ball  $ \mathbb{D}^{n} $ and $\Gamma$ be a Kleinian group acting on $ \mathbb{D}^{n} $. The orbital set $O _{\Gamma}(E) $ of $E$ under $ \Gamma $ is defined as the orbit of  $ E $ under the action of $\Gamma$, that is, 
	$$O _{\Gamma}(E)  =\bigcup_{g\in \Gamma}g(E). $$ 
	It was proved in \cite[Theorem 2.1]{ref5}  that if $ E $ is bounded in the hyperbolic metric, then
	$$\overline{\dim}_{\mathrm{B}}O _{\Gamma}(E)  =\max\left\lbrace\overline{\dim}_{\mathrm{B}}E, ~\overline{\dim}_{\mathrm{B}}L( \Gamma) ,~\delta(\Gamma)  \right\rbrace ,$$
	where $ L\left( \Gamma\right)  $ is the usual  limit set of $ \Gamma $  and $\delta\left( \Gamma\right)   $ is the Poincaré exponent of $ \Gamma $. Notably, this result does not require any assumption on $\Gamma$ and all three terms appearing in the maximum are necessary. However,  they constructed an example showing that this result may fail if $ E $ is unbounded, see  \cite[Theorem 2.2]{ref5}.

	\subsection{Rational maps}
	
	Motivated by the orbital sets in the context  of IFSs and Kleinian group actions, we turn our attention to complex dynamics. In this section we briefly recall the relevant dynamical concepts, but the orbital sets are defined in the following section. We refer the reader to \cite{ref1, mcmullen} for more background on rational maps in general.
	
	A rational map $ T $ of a complex variable $ z $ is a map  of the form  
	$$T( z) =\dfrac{P( z) }{Q( z) }, $$
	where $ P( z)  $ and $ Q( z)  $  are non-trivial coprime polynomials. We define the \textit{degree} of a rational map $ T   $ by 
	$$\deg T =\max \left\lbrace \deg  P  ,\deg  Q  \right\rbrace .$$
	In order to avoid degenerate cases, throughout this paper we always assume that $ \deg T=d \geq2 .$  
	
Rational maps act naturally on the Riemann sphere $ \mathbb{C}_{\infty} :=\mathbb{C} \cup \{\infty\}$ but by standard reductions we will see later that it is safe to assume everything happens in the complex plane $\mathbb{C}$. 	The maximal open subset of $ \mathbb{C}_{\infty}$  on which the family $ \left\lbrace  T^{n}: n\geq 1\right\rbrace  $ is normal, is called the \textit{Fatou set}, denoted by $ F_{T}  $ or $ F $. The complement of the Fatou set in $ \mathbb{C_{\infty}} $ is called the \textit{Julia set} of $ T ,$ denoted by $ J_{T} $ or $J$. The Julia set is often a beautiful fractal set with interesting dynamical and dimension theoretic properties. In particular, the Julia set is totally invariant, that is, $T(J) = T^{-1}(J) = J$.
	
	We define an equivalence relation on   $ \mathbb{C_{\infty}} $ as follows: for any $ z $ and $ w $ in $ \mathbb{C_{\infty}} $, we write $ z \sim w $ if and only if there exist non-negative integers $ n $ and $ m $ such that $ T^{n}( z) =T^{m}( w)  .$ The equivalence class containing $ z $, denoted by $ [ z]   $, is called the \emph{orbit} of $ z $. If $ [ z] $ is finite, then $ z $ is called an \emph{exceptional point} of $ T. $  According to \cite[Theorem 4.1.2]{ref1}, a rational map of degree at least 2 has at most 2 exceptional points.  Furthermore,     every exceptional point lies in the Fatou set, see\cite[Corollary 4.1.3]{ref1}.
	
	Two rational maps $ T $ and $ S $ are said to be conjugate if there exists a M{\"o}bius transformation $ g $ such that $ S=g\circ T\circ g^{-1} $. In this case, the Fatou set and Julia set are preserved under conjugacy, that is, $ F_{S} =g(F_{T})  $ and $ J_{S} =g(J_{T} )   .$ As a consequence, when the Julia set is not the whole Riemann sphere, we may always assume (up to conjugation) that the Julia set $ J $ is bounded in $ \mathbb{C} $. Indeed, from \cite[Theorem 4.3.2]{ref1}, we know that the Julia set $ J $ of a rational map $ T $ equals the whole Riemann sphere $ \mathbb{C_{\infty}} $ if and only if there exists some point $ z $ whose forward orbit $ \left\lbrace T^{n}( z) :n\geq 1\right\rbrace  $ is dense in the Riemann sphere. For example, the Julia set of the rational map $ T(z)=\frac{(z-2)^{2}}{z^{2}} $ is $ \mathbb{C_{\infty}} .$ In fact, there are many functions having the same property, such as the family $ T_{\omega}(z)=1+\frac{\omega}{z^{2}} $, where $ \omega \in \mathbb{C}^{*}= \mathbb{C}\setminus\left\lbrace 0\right\rbrace $; for more details see \cite[\S2.4]{ref10}. Our results trivially hold in the case when the Julia set is the whole of  $ \mathbb{C_{\infty}}$ and so again we may ignore this case throughout.
	
	A point $ z\in \mathbb{C_{\infty}} $ is called a \textit{critical point} of $ T $ if $ T $ fails to be injective in any neighborhood of $ z $. If $ z $ is a critical point of $ T $, then the value $ w=T( z)  $ is called a \textit{critical value} of $ T .$ Let $ \mathrm{Crit}( T) $ denote the set of all critical points of $ T $, and define the \emph{post critical set} of $T$ as 
	$$P\left( T\right)=\overline{\bigcup_{n=1}^{\infty}T^{n}( \mathrm{Crit}( T)) }.$$
The post critical set has a significant influence on the dynamics of $T$ and on the structure of the Julia set.  For example, if $P(T)$ is finite, then we say $T$ is \emph{post critically finite} and this is a particularly simple case.  More general is  when $J_T \cap P(T) = \emptyset$, in which case $T$ is \emph{hyperbolic}.  More general still is the case  when $J_T \cap P(T)$ is finite, in which case we say $T$ is \emph{geometrically finite}.  This includes all hyperbolic $T$ and also \emph{parabolic} $T$, that is, when $J_T$ contains a rationally indifferent periodic point, but no  critical points.

	\section{\textbf{Main results: dimensions of orbital sets for rational maps}}
	In this section, we formally  introduce the orbital set of a rational map and state our main results.
	
	Let $ E \subseteq \mathbb{C_{\infty}} $ be a non-empty  compact set and let $T$ be a rational map. We define the associated \textit{orbital set} $ O_{T}(E)  $   by
	$$ O_{T}(E)=\bigcup_{k=0}^{\infty} T^{-k}( E). $$
	When $ E=\left\lbrace z\right\rbrace  $, it reduces to the backward orbit of $ z ,$ denoted $ O_{T}(z) .$ We assume without loss of generality that the closure of  $O_T(E)$ is not the whole Riemann sphere and we may therefore assume (by conjugating if necessary) that $O_T(E)$ is bounded in $\mathbb{C}$.
	
	If $ z $ is a non-exceptional point of $ T $, then the Julia set $ J $ is contained in the closure of the backward orbit of $ z .$ In particular, if $ z \in J $, then the Julia set $ J $ is the closure of the backward orbit of $ z ,$ see\cite[Theorem 4.2.7]{ref1}. Hence, if $ E\subseteq \mathbb{C_{\infty}} $ is a non-empty compact set containing at least one non-exceptional point, then
	$J\subseteq \overline{O_{T}(E) }.$ As such, the orbital set shares geometric features of both the (dynamically defined)
	Julia set and the (non-dynamically defined) set $ E .$  The orbital set is typically a beautiful fractal set and we have produced several examples below with different choices of rational maps and different  choices of   $E$, including: circles, the Sierpi\'{n}ski triangle and the Vicsek fractal, see Figure \ref{fig:1} and Figure \ref{fig:2}. In particular, it will be easy to verify that the assumptions of Theorem \ref{thm:A} below are satisfied for all of the examples depicted in  Figure \ref{fig:1}.

Our  aim is to understand the dimension theory of the  orbital sets $ O_{T}(E)$, especially the upper box dimension. The following is our main theorem and we defer the proof until subsection \ref{proof:1}.
\begin{theorem}\label{thm:A}
	Let $ T $ be a rational map of degree at least 2 with $J=J_T$ a bounded subset of $\mathbb{C}$  and $ E $ be a non-empty compact subset of   $ \mathbb{C}$ satisfying the following property:\\
	 	(A) There exists a connected open set $U$ such that $E \subseteq U$,   $U \cap P(T) = \emptyset$, and $U \cap J  \neq \emptyset  $.\\
	Then 
	$$\overline{\dim}_{\mathrm{B}}O_{T}(E) =\max\left\lbrace\overline{\dim}_{\mathrm{B}}E, ~\overline{\dim}_{\mathrm{B}}J \right\rbrace. $$
	\end{theorem}
	
	On the way to proving Theorem \ref{thm:A} we are lead to a general dimension estimate for $J$ which may be of interest in its own right, see Proposition \ref{pros4.3}.

	It is worth making a few remarks about our assumptions in Theorem \ref{thm:A}.  First, analogous to the Kleinian case, it is noteworthy that we assume only very mild conditions on $T$.  For example, there is no requirement that $T$ satisfies any hyperbolicity conditions and no requirement that $T$ is geometrically finite. This makes our result rather general.  Second, our assumption (A) can be thought of in a similar way to the assumption that $E$ is bounded in the hyperbolic metric in the Kleinian case.  In particular, both assumptions are of the form: $E$ must be bounded away from regions where the dynamics become more complicated.  These assumptions are not very restrictive and in practice  just mean  that $E$ must avoid certain problematic locations.  In particular, the only case we really cannot handle at all is when $J \subseteq P(T)$, which is possible but relatively rare.  Indeed, if this is not the case, then there exists a simply connected open set $U$ such that $U \cap P(T) = \emptyset$ and $U \cap J  \neq \emptyset  $ and we may position $E$ inside $U$.  Further, if the complement of $P(T)$ is simply connected and $P(T)$ does not contain $J$, then (A) can be simplified to only requiring $E \cap P(T) = \emptyset$.   This will be the case in the geometrically finite setting and also when $P(T)$ is totally disconnected and does not contain $J$. Another case we can handle is when the complement of $P(T)$ has finitely many connected components, each of which intersects $J$. In this case we can handle any $E$ which avoids $P(T)$, possibly by splitting it up into finitely many pieces and appealing to the finite stability of upper box dimension.

	We also show, analogous to the Kleinian setting, that assumption (A) cannot be removed in general (or weakened in certain natural ways), see Proposition \ref{thm:B} below. That said, it is easy to come up with examples for which our result remains true but which do not satisfy the assumptions from the theorem.  For example, we may add to $E$  the orbit of an exceptional point which only adds a finite number of points to the orbital set but necessarily violates assumption (A), see Lemma \ref{bimpliesa}.

The formula from Theorem \ref{thm:A} parallels the analogous dimension formula for Kleinian groups. This correspondence reflects the unified viewpoint provided by the Sullivan dictionary, linking the dynamics of rational maps and Kleinian groups. In particular,  we can also call 
$$\overline{\dim}_{\mathrm{B}}O_{T}(E) =\max\left\lbrace\overline{\dim}_{\mathrm{B}}E, ~\overline{\dim}_{\mathrm{B}}J \right\rbrace $$
the `expected formula' in complex dynamics, noting however, that it does not always holds,  see Proposition \ref{thm:B} below.

 Theorem \ref{thm:A} easily gives the following corollary concerning backward orbits of single points, which may be of interest in its own right.
\begin{corollary} \label{cor1}
	Let $ T $ be a rational map of degree at least 2 with $J=J_T$ a bounded subset of $\mathbb{C}$ and let $ z\in \mathbb{C}$ be a point with bounded backwards orbit $O_T(z) \subseteq \mathbb{C}$ such that  the following property holds:\\
	(A') There exists a  connected open set $U$ such that $z \in U$,   $U \cap P(T) = \emptyset$, and $U \cap J  \neq \emptyset  $.\\
	Then 
	$$\overline{\dim}_{\mathrm{B}}O_{T}( z) = \overline{\dim}_{\mathrm{B}}J . $$
	
\end{corollary}
Whilst we focus on the upper box dimension, Theorem \ref{thm:A} and Corollary \ref{cor1} provide sufficient conditions for equality of upper and lower box dimensions due to the easy reverse inequality
\[
\underline{\dim}_{\mathrm{B}}O_{T}(E) \geq \max\left\lbrace\underline{\dim}_{\mathrm{B}}E, ~\underline{\dim}_{\mathrm{B}}J \right\rbrace
\]
which holds by monotonicity of lower box dimension. For example, if the box dimensions of both $J$ and $E$ exist, then the box dimension of the orbital set exists.  In general, if the box dimension of $E$ does not exist, then estimating the lower box dimension of the orbital set is rather more subtle and will likely depend on finer scaling information regarding $E$.  This was shown to be the case  for inhomogeneous self-similar sets in \cite{ref6} with more precise information   recently obtained  in \cite{rutar}.

Next we show that if  assumption (A) of Theorem \ref{thm:A} is removed, then   the `expected formula' can fail to hold, even for extremely simple $T$ and $E$.

\begin{proposition}\label{thm:B} There are several examples where the expected formula fails to hold, including:\\
	(i) Let $T$ be defined by $T(z) = z^2$ and   $ E=[0,1] \subseteq  \mathbb{C}$.  Then   assumption  (A)   from Theorem~\ref{thm:A} fails   and, indeed,
	$$\overline{\dim}_{\mathrm{B}}O_{T}(E)= \dim_{\mathrm{H}}\overline{O_{T}(E)}=2,$$
	and 
	$$\overline{\dim}_{\mathrm{B}}E = \dim _{\mathrm{H}}J = \overline{\dim}_{\mathrm{B}}J=1.$$ 
	(ii) There exists a non-empty compact set $ E\subseteq\mathbb{C} $ and a hyperbolic  quadratic polynomial $ T $   for which assumption  (A) of Theorem~\ref{thm:A} fails and
	$$\overline{\dim}_{\mathrm{B}}O_{T}(E)=\frac{2}{3},$$
	and 
	$$\max\{\overline{\dim}_{\mathrm{B}}E, \overline{\dim}_{\mathrm{B}}J\} < 2/3.$$
	Moreover,   $E$ still satisfies the conclusion of Lemma \ref{bimpliesa} (that is, the orbital set accumulates on $J$), which is notable because this does not hold for example (i).  \\
	(iii) By modifying the previous example we get that for all $\varepsilon>0$, there exists a non-empty compact set $ E\subseteq\mathbb{C} $ and a polynomial $ T $   for which 
	$$\overline{\dim}_{\mathrm{B}}O_{T}(E)>2-\varepsilon,$$
	and 
	$$\max\{\overline{\dim}_{\mathrm{B}}E, \overline{\dim}_{\mathrm{B}}J\} < \varepsilon$$ 
	and so there is no limit to how badly the  conclusion from Theorem \ref{thm:A} can fail. \\
	(iv) If $T$ admits a Siegel disk or Herman ring, then there exists a choice of $E$ with $$\overline{\dim}_{\mathrm{B}}E=1$$ for which 
	$$\overline{\dim}_{\mathrm{B}}O_{T}(E)=2.$$
	In particular, there are examples of such $T$ with   $$ \overline{\dim}_{\mathrm{B}}J  <2$$
	and so the conclusion of Theorem \ref{thm:A} fails. Moreover these examples fail   assumption  (A)  from Theorem \ref{thm:A} but they would satisfy it if $U$ was not required to intersect $J$.
\end{proposition}
We defer the proof of Proposition \ref{thm:B} until subsection \ref{proof:3}. We are able to compute the upper box dimension of $O_T(E)$ in example (ii) from Proposition \ref{thm:B} by using the fact that $T^{-1}(E)$ \emph{does} satisfy assumption (A) from Theorem \ref{thm:A}, despite the fact that $E$ does not.  This phenomena provides another general class we can handle.  Indeed, if $T^{-k}(E)$ satisfies assumption (A) for some $k \geq 1$, then 
\[
O_T(E) = E \cup T^{-1}(E) \cup \cdots \cup T^{-(k-1)}(E) \cup O_T(T^{-k}(E))
\]
and then
	$$\overline{\dim}_{\mathrm{B}}O_{T}(E) =\max\left\lbrace\overline{\dim}_{\mathrm{B}}E, \ \overline{\dim}_{\mathrm{B}}T^{-1}(E), \dots, \overline{\dim}_{\mathrm{B}}T^{-k}(E),  \ ~\overline{\dim}_{\mathrm{B}}J \right\rbrace $$
	by Theorem \ref{thm:A} and finite stability of the upper box dimension.

Finally, we turn our attention to the Hausdorff dimension.  Due to it being countably stable, it is much easier to handle, but still requires a little work.  We get the  following simple result  and  defer the proof until subsection \ref{proof:2}.
\begin{proposition}\label{thm: C}
	Let $ T $ be a rational map of degree at least 2 and $ E $ be a non-empty compact subset of the Riemann sphere $ \mathbb{C_{\infty}} $.  
	Then 
	$$\dim_{\mathrm{H}}O_{T}(E)=\dim_{\mathrm{H}}E.$$
\end{proposition}
One might wonder why Proposition \ref{thm: C} differs from the Hausdorff dimension formula for inhomogeneous attractors, since it does not include $\dim_{\mathrm{H}} J$.  However, this is because we did not take the closure of the orbital set and if we did then we would obtain
	\begin{align} \tag{2.1}\label{hdformula}
	\dim_{\mathrm{H}}\overline{O_{T}(E)}=\max\{ \dim_{\mathrm{H}}E,  \dim_{\mathrm{H}}J\}
	\end{align}
	provided that 
	\[
	\overline{O_{T}(E)} = O_{T}(E) \cup J.
	\]
	This latter condition is guaranteed provided that the assumptions of Lemma \ref{lem:1} hold for $A=E$.  Otherwise, \eqref{hdformula} may fail even though Proposition \ref{thm: C} holds; see Proposition \ref{thm:B} (i) for such an example.   We chose not to take the closure of the orbital set in the definition because taking the closure does not affect the box dimensions and so the results pertaining to upper box dimension remain unchanged.  

\begin{figure}[htp]
	\centering

	\begin{subfigure}[t]{0.32\textwidth}
		\includegraphics[width=\linewidth]{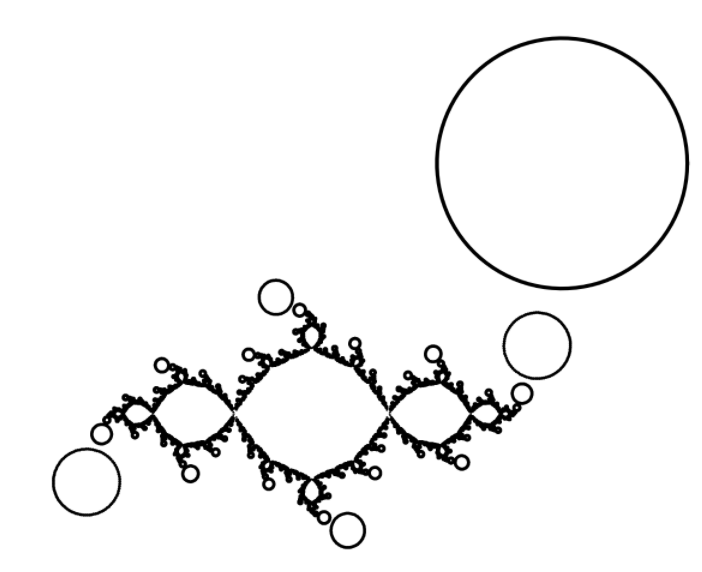}
		\caption{Orbital set generated by the rational map $T_{1}(z) = z^2 - 1$ and $E = \{z: |z-2-2i|=1\}$.}
		\label{fig:sub2}
	\end{subfigure}
	\hfill
	\begin{subfigure}[t]{0.32\textwidth}
		\includegraphics[width=\linewidth]{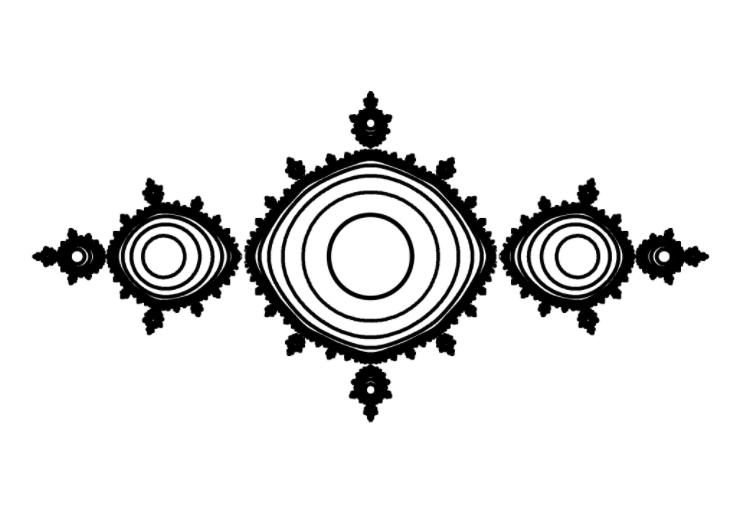}
		\caption{Orbital set generated by the rational map $T_{1}(z) = z^2 - 1$ and $E=\left\lbrace z: |z|=0.2\right\rbrace $} 
	\label{fig:sub3}
\end{subfigure}
\hfill
\begin{subfigure}[t]{0.32\textwidth}
	\includegraphics[width=\linewidth]{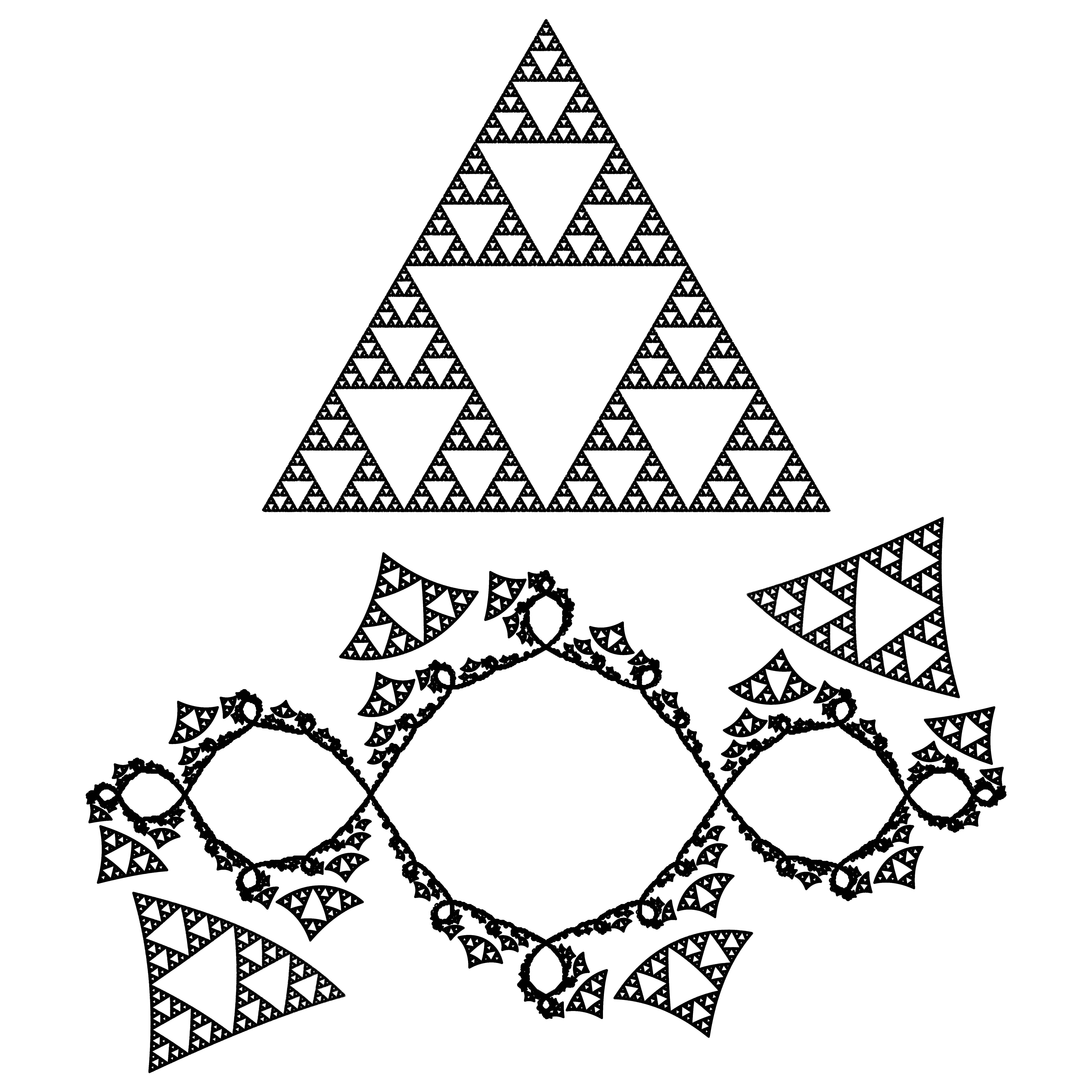}
	\caption{Orbital set generated by the rational map $T_{1}(z) = z^2 - 1$ and a Sierpiński triangle.}
	\label{fig:sub5}
\end{subfigure}

\vspace{35pt}

\begin{subfigure}[t]{0.32\textwidth}
	\includegraphics[width=\linewidth]{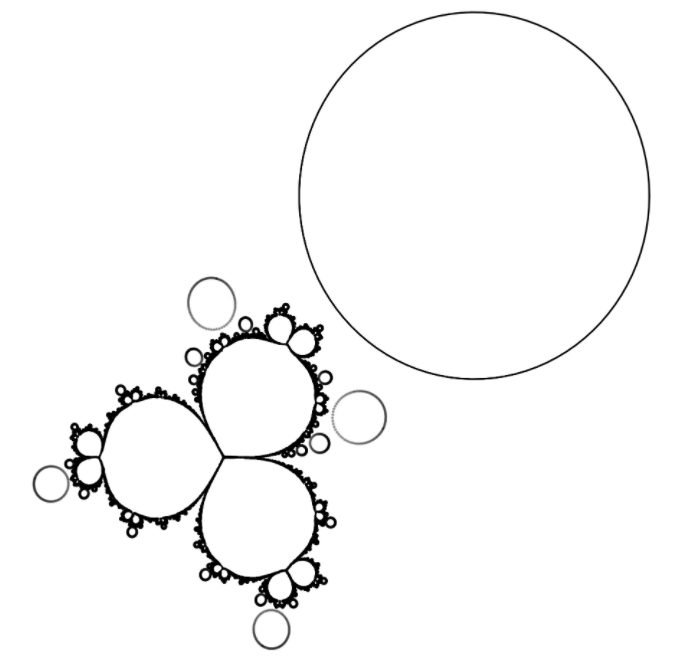}
	\caption{Orbital set generated by the rational map $T_{2}(z) =  z^4+z$ and $E = \{z: |z-2-2i| = 1.4\}$.}
	\label{fig:sub4}
\end{subfigure}
\hfill
\begin{subfigure}[t]{0.32\textwidth}
	\includegraphics[width=\linewidth]{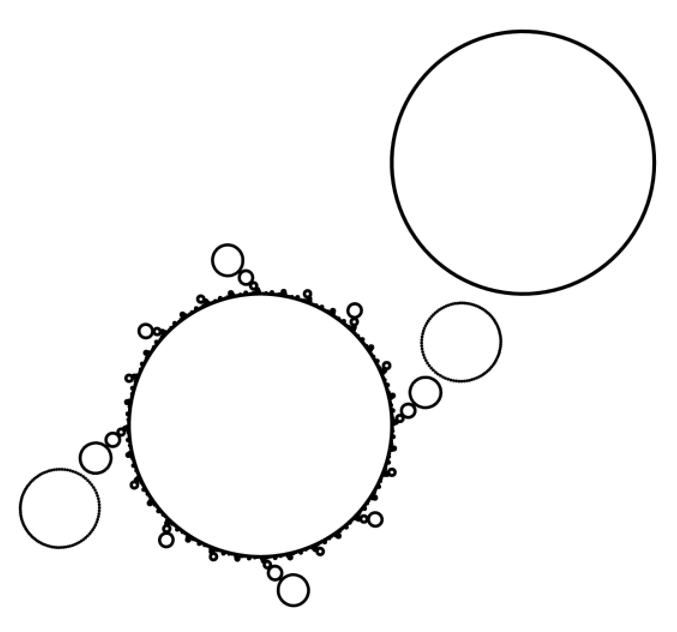}
	\caption{Orbital set generated by the rational map $T_{3}(z)=z^{2}$ and $E=\{z:|z-2-2i|=1\}$.}
	\label{fig:sub1}
\end{subfigure}
\hfill
\begin{subfigure}[t]{0.32\textwidth}
	\includegraphics[width=\linewidth]{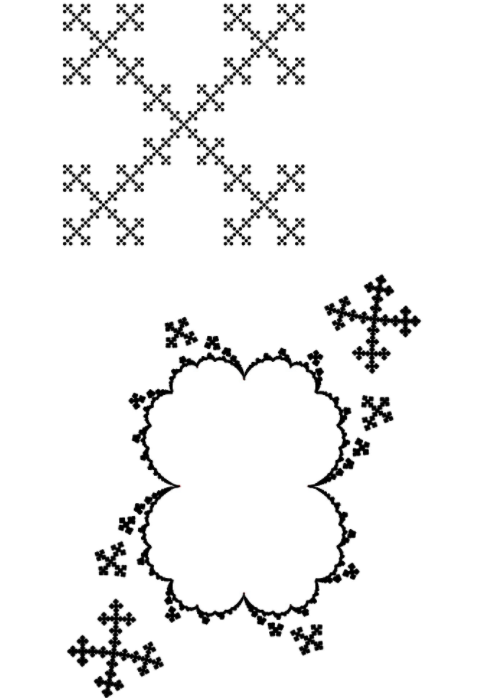}
	\caption{Orbital set generated by the rational map $T_{4}(z) = z^2+z$ and a Vicsek fractal.}
	\label{fig:sub6}
\end{subfigure}

\caption{Orbital sets generated by standard rational maps and various compact subsets of $ \mathbb{C_{\infty}} $. The maps include $ T_{1}( z) =z^{2}-1,~T_{2}( z) =z^{4}+z, ~T_{3}( z) =z^{2} $ and $T_{4}(z) = z^2+z  $ and the compact subsets include circles and translates of the Sierpiński triangle and the Vicsek fractal. }
\label{fig:1}
\end{figure}
\begin{figure}[htp]
	\centering

	\begin{subfigure}[t]{0.48\textwidth}
		\includegraphics[width=\linewidth]{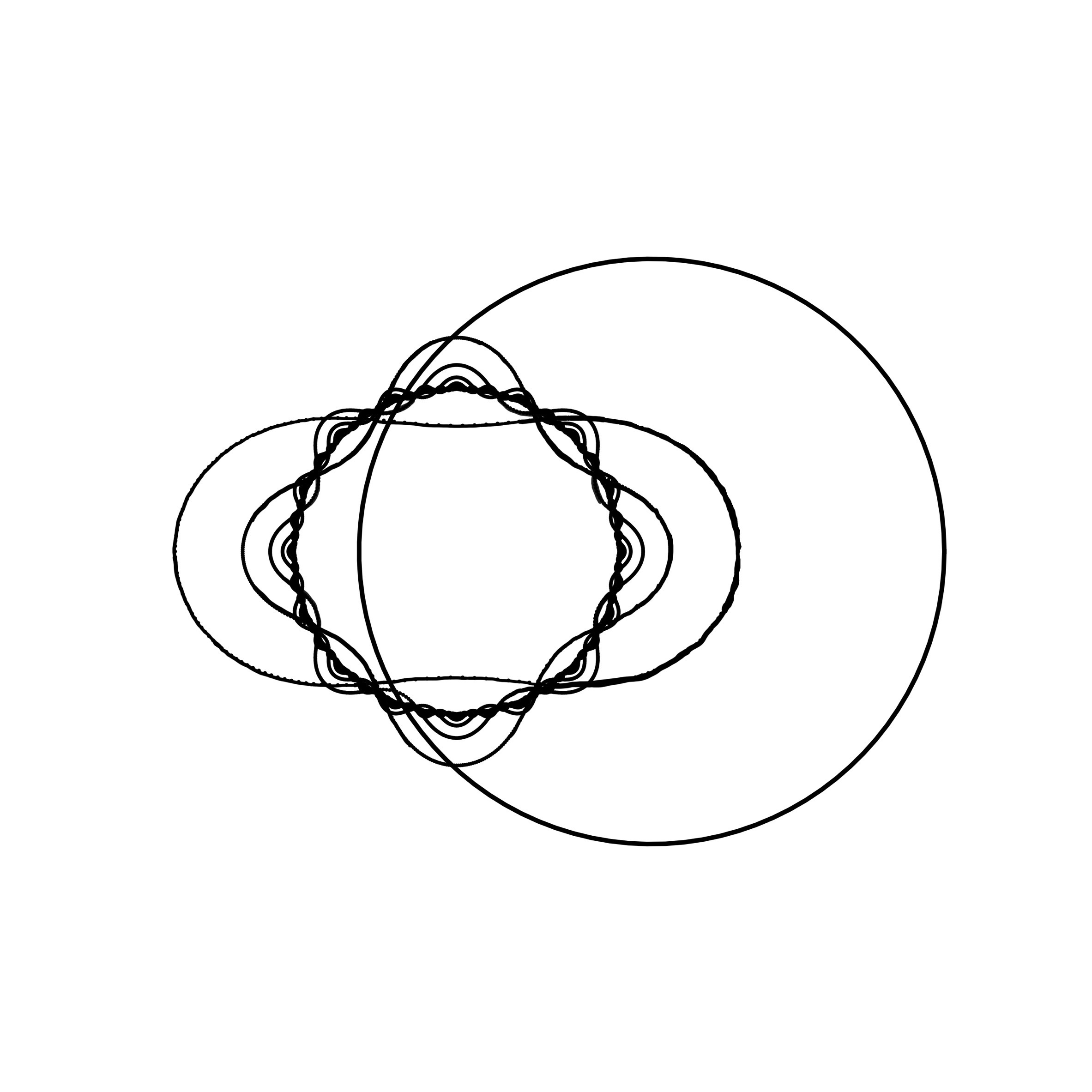}
		\caption{Orbital set generated by the rational map $T_{1}(z) = z^2 $ and $E = \{z: |z-1.2|=1.8\}$.}
		\label{fig:sub2_1}
	\end{subfigure}
	\hfill
	\begin{subfigure}[t]{0.48\textwidth}
		\includegraphics[width=\linewidth]{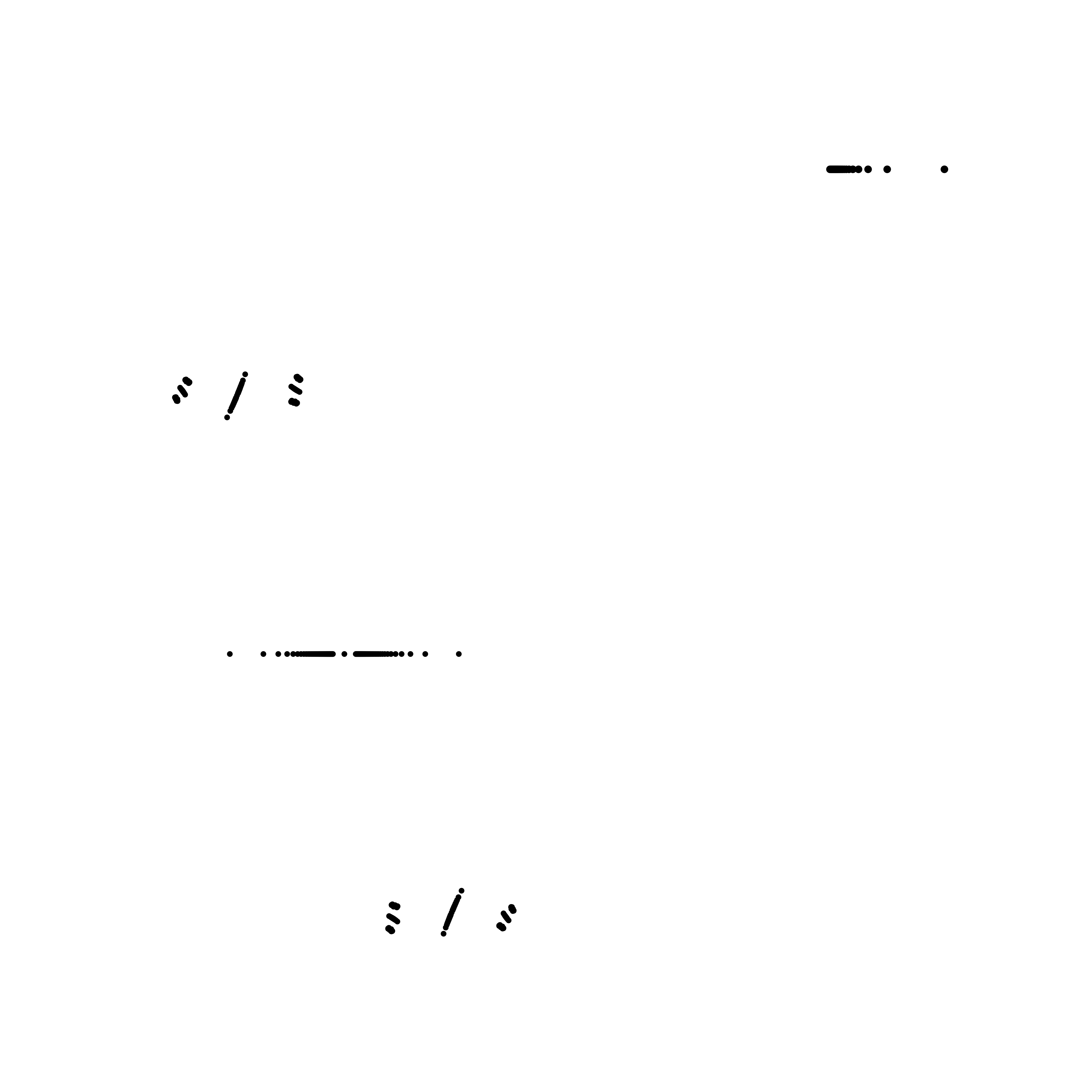}
		\caption{Orbital set generated by the rational map $T_{2}(z)=z^{2}+3\sqrt{2}+3\sqrt{2}i$ and $E=\{3\sqrt{2}+3\sqrt{2}i\}\cup\{3\sqrt{2}+3\sqrt{2}i+\frac{1}{n}\}_{n\in \mathbb{N}}$. This is similar to the example from Proposition \ref{thm:B} (ii) which fails Assumption (A).}
		\label{fig:sub2_2}
		
	\end{subfigure}
	
	\vspace{35pt}

	\begin{subfigure}[t]{0.48\textwidth}
		\includegraphics[width=\linewidth]{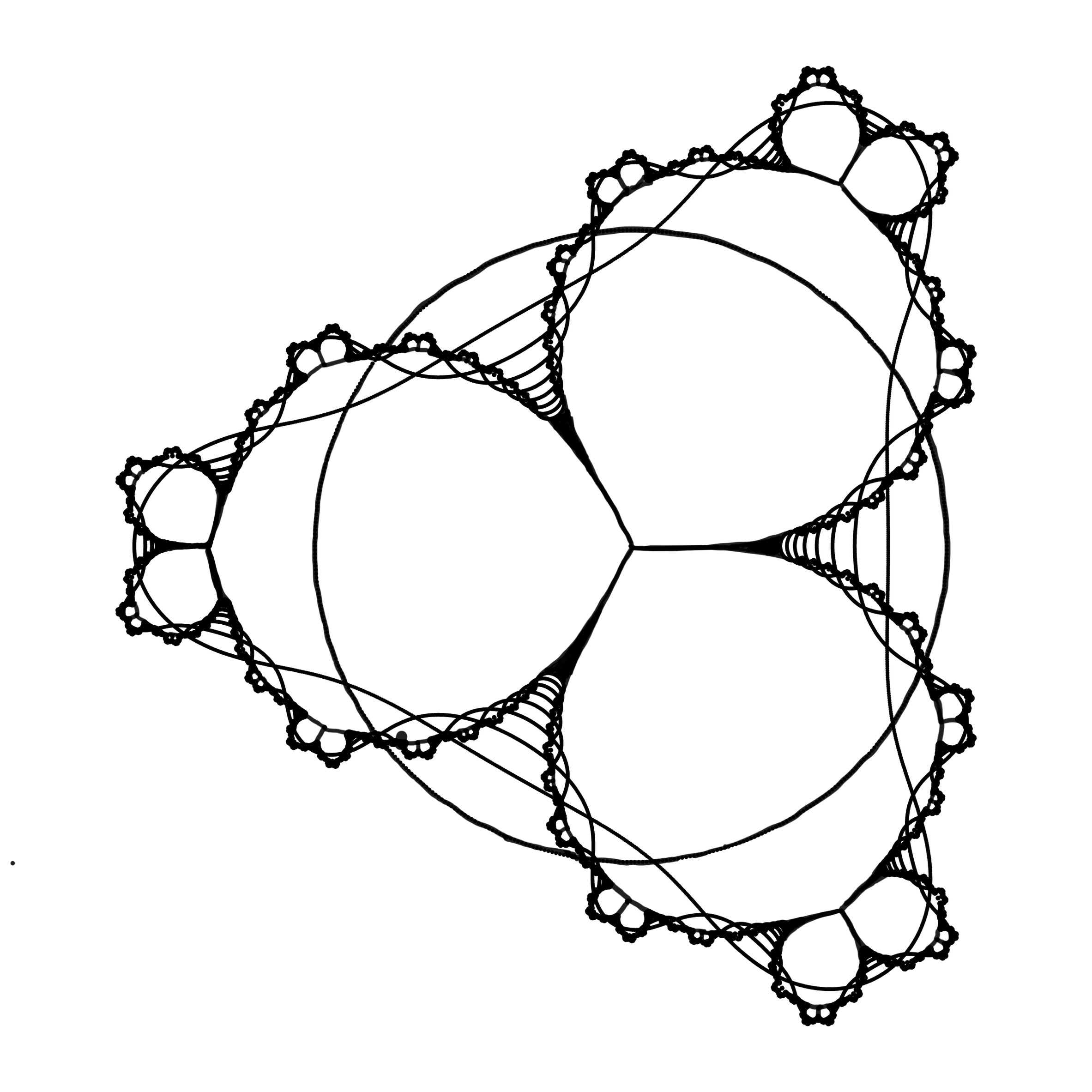}
		\caption{Orbital set generated by the rational map $T_{3}(z) =  z^4+z$ and $E = \{z: |z| = 0.75\}$.}
		\label{fig:sub2_3}
	\end{subfigure}
	\hfill
	\begin{subfigure}[t]{0.48\textwidth}
		\includegraphics[width=\linewidth]{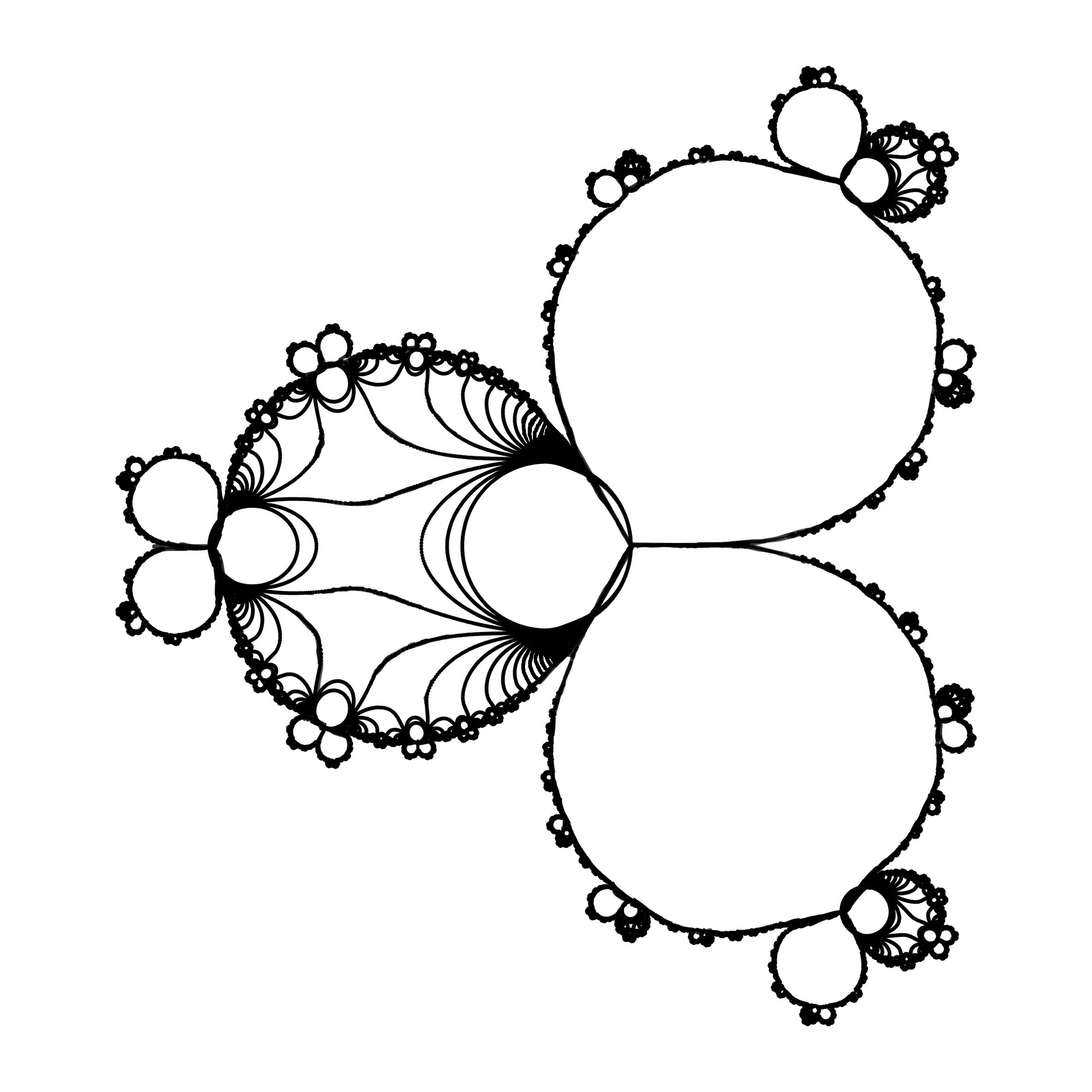}
		\caption{Orbital set generated by the rational map $T_{3}(z) = z^4+z$ and $E=\left\lbrace z: |z+0.2|=0.2\right\rbrace $. This example fails  Assumption (A).} 
	\label{fig:sub2_4}
\end{subfigure}

\caption{Orbital sets generated by standard rational maps and various compact subsets of $ \mathbb{C_{\infty}} $. The maps include $ T_{1}(z) = z^2,~T_{2}(z)=z^{2}+3\sqrt{2}+3\sqrt{2}i,$ and $T_{3}(z) =  z^4+z  $ and the compact subsets are different circles and $\{3\sqrt{2}+3\sqrt{2}i\}\cup\{3\sqrt{2}+3\sqrt{2}i+\frac{1}{n}\}_{n\in \mathbb{N}}$, where $3\sqrt{2}+3\sqrt{2}i  $ is a critical value of $ T_{2}(z) $. }
\label{fig:2}
\end{figure}

\newpage
\section{\textbf{Proofs}}

\subsection{Notation, first reductions, and inverse branches} \label{reductions}

For $A,B \geq 0$, we write $ A\lesssim B $ if there exists a constant $ c>0 $ such that $ A\leq cB, $ and $ A\approx B $ if  $ A\gtrsim B $ and $ A\lesssim B .$ We write $\#X$ to denote the cardinality of a set $X$. We also write  $ \textup{dist}\left( x, X\right)  $ to denote the distance between a point $ x $ and a set $ X\subseteq\mathbb{C_{\infty}}, $ defined by 
$$\textup{dist}\left( x, X\right)=\inf\left\lbrace | x-y| :y\in X\right\rbrace .$$
Similarly, $ \textup{diam}\left( X\right)  $ means the diameter of the set $ X $, defined by 
$$\textup{diam}\left( X\right)=\sup\left\lbrace | x-y| :x,y\in X \right\rbrace .$$ 

Throughout this section up until the proof of Theorem \ref{thm:A}, we fix  a rational map  $ T $  of degree at least 2 with $J=J_T$ a bounded subset of $\mathbb{C}$ and a non-empty compact set $ E\subseteq \mathbb{C} $ and let $U$ be a  connected open set $U$ such that $E \subseteq U$, and such that $U \cap P(T) = \emptyset$ and $U \cap J  \neq \emptyset  $.  The existence of $U$ is guaranteed by assumption (A) in Theorem \ref{thm:A}.

Without loss of generality, we   assume that   $ U  $ is   \emph{simply} connected and this   reduction is   important.  We explain how to make this reduction here. Since $U$ is originally assumed to be open and connected, it is also path connected. Fix $z \in (J \cap U) \setminus P(T)$, noting that we can choose such $z$ by assumption (A).  To each point $x \in E$, let $U_x, U_x'$ be   open neighbourhoods of a path connecting $z$ and $x$ with the property that $\overline{U_x} \subseteq U_x' \subseteq U$ and with both $U_x, U_x'$   simply connected.     We can do this because  $E$ is compact and by assuming that the paths connecting $x$ and $z$ are simple curves ($U$ must also be arc connected).  The open sets $U_x$ form a cover of $E \cup \{z\}$ and since $E \cup \{z\}$ is compact we may extract a finite subcover $\{U_x\}_{x \in A}$ with $A$ finite.  For each $x \in A$, define $E_x =\overline{E \cap U_x}$ and note that $E_x$ now satisfies assumption (A) with the   \emph{simply connected} open set $U_x'$.   Finally, since the upper  box dimension is finitely stable we get
$$\overline{\dim}_{\mathrm{B}}O_{T}( E) =\max_{x \in A}\left\lbrace\overline{\dim}_{\mathrm{B}}O_{T}( E_x) \right\rbrace$$
and it suffices to consider each $ E_x$ separately and we have therefore  made the desired reduction.

Let $  d = \deg T $. If $ w $ is not a critical value of the rational map $ T^{n}$ ($n \geq 1$), then there exist $ d^{n}$ distinct locally defined analytic inverse branches $ S_{n,j} $, $ j = 1, \dots, d^{n} $, of $ (T^{n})^{-1}$ defined in a neighbourhood of $ w $, see \cite[Chapter 9, Introductory Remarks]{ref1}. 
Each $ S_{n,j} $ is a single-valued analytic function defined near $ w $, satisfying $S_{n,j}(w) \in (T^{n})^{-1}(w) $, and these values are distinct for distinct $ j. $  Moreover, since we assume that $ E $ is contained in the simply connected domain $ U $  which contains no critical values of the rational map $ T^{n}$,  by the monodromy theorem \cite[Chapter 8, Theorem 2]{ref3}, each inverse branch $ S_{n,j} $ can be uniquely analytically continued to all of $ U $, and is therefore well-defined on $ U .$ We define the inverse branches on the whole of $ U $ because this allows us to ensure univalency and to apply the Koebe distortion theorem on a neighbourhood of $ E $ rather than just on $ E $ itself. This is crucial for controlling derivatives and for estimating the thermodynamic properties of the orbital set. Since $ S_{n,j} $ is locally injective,   injectivity is preserved under analytic continuation in $ U $. Therefore, each $ S_{n,j} $ is univalent on the whole of $ U $. Throughout the rest of the paper, we write
$$|S'_{n,j}(E)| := \max_{z\in E}|S'_{n,j}(z)|,  \qquad (\text{for $n \geq 1$ and }j=1,\dots,d^{n}),$$
where the maximum value is attained by the compactness of $ E $. Further note that $0<|S'_{n,j}(E)|<\infty$ for all $n,j$. For convenience we also write $S_{0,1} = T^0$ both to denote the identity map.

\subsection{Dynamical properties of rational maps}

In this section we recall some important results from complex dynamics and geometric function theory which we will need.  The following result characterizes the critical values of the iterates.
\begin{lemma}\cite[Theorem 2.7.3]{ref1}
	Let $\mathrm{Crit}( T) $ be the set of critical points of  $ T $. Then the set of critical values of $ T^{n}$ is given by
	$$\bigcup_{k=1}^{n}T^{k}( \mathrm{Crit}( T)) .$$
	In particular, for each $n \geq 1$, this is a subset of the post critical set $P(T)$.
\end{lemma}

If $ z \in F $ and is not exceptional, then the closure of its backward orbit contains $ J .$  The following result makes this uniform over $z \in E$. We state it with a general compact set $A$ but will apply it with $A=E$, see Lemma \ref{bimpliesa}.

\begin{lemma}\label{lem:1}\cite[Theorem 4.2.8]{ref1}
	Let $A \subseteq \mathbb{C}_\infty$ be a compact set with the property that for all $ z $ in $ F $, the sequence $ \left\lbrace T^{n}( z) :n\geq1\right\rbrace  $ does not accumulate at any point of $A$. Then, for any open set $ V $ which contains $ J $, $ T^{-n}( A) \subseteq V $ for all sufficiently large $ n. $
\end{lemma}

To estimate the local behaviour of the inverse branches of a rational map, we need the following results, which are called the Koebe One-Quarter Theorem and Generalized Distortion Theorem respectively.
\begin{lemma}\cite[Theorem 14.7.8]{ref16}\label{lem:4}
	Let $ f $ be an analytic and univalent function in the unit disk $ \mathbb{D}^{2} $ satisfying $ f(0)=0 $ and $ f'(0)=1 $. Then the disk $ \left\lbrace w:|w|<\frac{1}{4} \right\rbrace  $ is contained in $ f( \mathbb{D}^{2} ) $.
\end{lemma}
\begin{lemma}\cite[Theorem 14.7.16]{ref16}\label{lem:2}
	If $ K $ is a compact subset of an open region $ \Omega $, there is a constant $ M \left( dependent~ on~ K \right) $ such that for every univalent function $ f $ on $ \Omega $ and every pair of points $ z $ and $ w $ in $ K $,
	$$\dfrac{1}{M}\leq \dfrac{|f'( z) |}{|f'( w) |} \leq M.$$
\end{lemma}

\subsection{Preliminary estimates}	

In this section, we prove several key technical lemmas to support our proof of Theorem \ref{thm:A}. The following lemma is well-known but the proof is often omitted in the literature and so for completeness we include  its simple deduction from Lemma \ref{lem:4}.
\begin{lemma}\label{pros4.4}
	Let $ f $ be a univalent map from $ \Omega $ to $ \Omega' $, where $ \Omega $ and $ \Omega' $ are simply connected domains. Then for any point $ z \in \Omega $, we have
	$$\dfrac{1}{4}\textup{dist}\left( f( z) ,\partial \Omega'\right) \leq |f'( z) | \textup{dist}\left( z,\partial\Omega \right) \leq 4 \textup{dist}\left( f( z) ,\partial \Omega'\right).$$
\end{lemma}
\begin{proof}
	Let $ r :=\textup{dist}\left( z,\partial\Omega \right)$, so that the disk $ D=\left\lbrace w:|w-z|<r\right\rbrace  $ is contained in $ \Omega $. Let $ g $ be the map from the unit disk $ \mathbb{D}^{2} $ to the disk $ D $ defined by 
	$$g(\theta)=z+r\theta.$$
	Then the map $ f\circ g $ is analytic and univalent on $ \mathbb{D}^{2} $ and 
	$$(f\circ g)(0)=f(z),\textup{ and } (f\circ g)'(0)=f'(z)r.$$
	By Lemma \ref{lem:4}, $ (f\circ g)(\mathbb{D}^{2}) $   contains a disk centred at $ f(z) $ with radii $ \frac{|f'(z)|r}{4} $, that is, 
	$$\textup{dist}\left( f(z),\partial (f\circ g)(\mathbb{D}^{2}) \right)\geq \dfrac{|f'(z)|r}{4}.$$
	Since $(f\circ g)(\mathbb{D}^{2})\subseteq \Omega'  $, we know
	$$\textup{dist}\left( f( z) ,\partial \Omega'\right)\geq\textup{dist}\left( f(z),\partial (f\circ g)(\mathbb{D}^{2}) \right)\geq \dfrac{|f'(z)|r}{4},$$
	that is,
	$$ |f'( z) | \textup{dist}\left( z,\partial\Omega \right) \leq 4 \textup{dist}\left( f( z) ,\partial \Omega'\right).$$
	Similarly, by applying the same argument to the inverse map $ f^{-1} $, we obtain
	$$\dfrac{1}{4}\textup{dist}\left( f( z) ,\partial \Omega'\right) \leq |f'( z) | \textup{dist}\left( z,\partial\Omega \right) $$
	as required.
\end{proof}
We will need the conclusion of Lemma \ref{lem:1} to hold for $A=E$ and the following lemma allows that.  This is the result which ensures the preimages of $E$ actually accumulate on $J$ and so the orbital set parallels the orbital sets previously studied in the IFS and Kleinian group settings. 
\begin{lemma}\label{bimpliesa}
For all  $ z $ in $ F $, the sequence $ \left\lbrace T^{n}( z) :n\geq1\right\rbrace  $ does not accumulate at any point of $E$.  In particular, for any open set $ V $ which contains $ J $, $ T^{-n}( E) \subseteq V $ for all sufficiently large $ n. $
\end{lemma}
\begin{proof}
Let  $ z $ in $ F $ and let $u$ be an accumulation point of the sequence $ \left\lbrace T^{n}( z) :n\geq1\right\rbrace  $.  Due to the classification of Fatou components \cite[Theorem 3.2]{mcmullen} there are only a few possibilities.  Indeed, $u$ must belong to an attracting cycle, a parabolic cycle, or must belong to a Siegel disk or a Herman ring.  By \cite[Corollary 3.7]{mcmullen} $P(T)$ contains all attracting or parabolic  cycles  and the boundary of every Siegel disk or Herman ring.  In particular, since $E \subseteq U$ and $U \cap P(T) = \emptyset$,  this means $E$ cannot intersect any attracting or parabolic cycles.  Therefore, if $u \in E$, the only possibility remaining is for it to belong to a Siegel disk or Herman ring.  However, since $U \cap J \neq \emptyset$, this is impossible because $P(T)$ contains the boundary of such components and so $U$ cannot cross this (connected) boundary to intersect $J$, recalling that $U$ is connected. The final conclusion follows by applying Lemma \ref{lem:1} with $A=E$.
\end{proof}

The next result allows us to control the covering number of a pre-image of $E$ by an appropriately rescaled covering number of $E$ itself.  
\begin{lemma}\label{pros4.1}
For integer  $n \geq 1$,   $ \delta \in \left( 0,1\right)$, and $j=1,\dots,d^{n}$,
		\begin{align}
		N_{\delta}(S_{n,j}( E) ) \lesssim N_{\frac{\delta}{ |S'_{n,j} ( E) | }}\left(  E\right),  \tag{3.1} \label{eq:main3.1}
	\end{align} 
	where the implicit constant is independent of $ n $, $j$ and $\delta$.
\end{lemma}
\begin{proof}
	For any $j=1,\dots,d^{n}  $, let $ \left\lbrace U_{\alpha}\right\rbrace_\alpha  $ be a $\frac{\delta}{ |S'_{n,j} ( E) | } $-cover of $ E $ such that $\overline{U_{\alpha}} \subseteq U' \subseteq U $ for some compact set $U'$ not depending on $n$ or $\delta$.  This can be achieved since $E \subseteq U$ with $E$ compact and $U$ open. Recall that $U$ is the fixed simply connected open set containing $E$, described in subsection \ref{reductions}.  We know 
	$$\textup{diam}\left(S_{n,j}( U_{\alpha}) \right) \leq \sup_{z\in \overline{U_{\alpha}}}|S'_{n,j}( z)|\cdot\textup{diam}\left(U_{\alpha}\right) .$$
	According to Lemma \ref{lem:2}, 
	$$\sup_{z\in \overline{U_{\alpha}}}|S'_{n,j}( z)|\approx \sup_{z\in \overline{U_{\alpha}}\cap E}|S'_{n,j}( z)| $$
	with the implicit constant depending on $U'$ but not on $n$ or $\delta$. 	Then
	$$\textup{diam}\left(S_{n,j}( U_{\alpha}) \right)\lesssim \sup_{z\in \overline{U_{\alpha}}\cap E}|S'_{n,j}( z)|\cdot\textup{diam}\left(U_{\alpha}\right)\leq |S'_{n,j}( E)  |\cdot\textup{diam}\left(U_{\alpha}\right) .$$
	As $ \left\lbrace U_{\alpha}\right\rbrace  $ is a $\frac{\delta}{ |S'_{n,j} ( E) | } $-cover of $ E ,$ then 
	$$\textup{diam}\left(S_{n,j}( U_{\alpha}) \right)\lesssim\delta.$$
	Therefore
	$$N_{\delta}(S_{n,j}( E) ) \lesssim N_{\frac{\delta}{ |S'_{n,j} ( E) | }}\left(  E\right), $$
	as required.	 
\end{proof}

The next result allows us to control the derivatives of inverse branches in terms of how close they map $E$ to the Julia set.  
\begin{lemma}\label{pros4.2}
	Let $ \delta \in \left( 0,1\right)  $ and suppose that for some integers $ n \geq 0 $ and $j=1, \dots, d^n$,  $ S_{n,j}(E)\nsubseteq J_{\delta} $. 	Then  
	\begin{align*}
		\frac{\delta}{ |S'_{n,j}( E) | }  \lesssim  1, \tag{3.2} 
	\end{align*} 
	where the implicit constant is independent of $ n $, $j$ and $\delta$.
\end{lemma}
\begin{proof}
	 By assumption (A) in Theorem \ref{thm:A}  there exists $y \in J \cap U $.   Let $U'$ be a fixed compact set with  $E \cup \{y\} \subseteq U' \subseteq U$, which exists since $U$ is open and $E \cup \{y\}$ is compact. 	Since $ S_{n,j}( E) \nsubseteq J_{\delta} ,$   there exists a point $ w_{0}\in E $  such that $ S_{n,j}( w_{0}) \notin J_{\delta}$. Since $ S_{n,j}( w_{0}) \notin J_{\delta}$ and using $T$-invariance of $J$ to get $S_{n,j}( y) \in J$, we have 
	\begin{align*}
	\delta\leq|S_{n,j}( w_{0}) -S_{n,j}( y)|&\leq\max_{w\in U' }| S'_{n,j}( w) |\cdot|w_{0}-y| \\
	&\approx |S'_{n,j}( E) | \cdot|w_{0}-y| \qquad \text{(by Lemma \ref{lem:2})}\\
	&\leq |S'_{n,j}( E) |\cdot\text{diam}\left( E\cup J\right).
	\end{align*} 			
	As $\text{diam}\left( E\cup J\right)  $ is finite, we conclude that
	$$\frac{\delta}{ |S'_{n,j}( E) | }  \lesssim  1$$
	as required. 		
\end{proof}

The next result is in a certain sense a converse to  Lemma \ref{pros4.2} because it ensures that if the image of $E$ under a particular inverse branch is very small, then it must be close to the Julia set.

\begin{lemma}\label{lem4.1}
	Let  $ z\in E\setminus J$ and  $ x\in T^{-n}(z) $ for some $n \geq 0$, and suppose that $ |(T^{n})'(x)|\geq\frac{1}{r} $ for some   $ r >0$. Then $ x $ is in the $ \lesssim r $ neighbourhood of $ J $.
\end{lemma}
\begin{proof}
	By assumption (A) in Theorem \ref{thm:A}  there exists $y \in J \cap U $.  Again, let $U'$ be a fixed compact set with  $E \cup \{y\} \subseteq U' \subseteq U$, which exists since $U$ is open and $E \cup \{y\}$ is compact.  Let $ x=S_{n,j}(z) $ for some $ j $, recalling that the inverse branches are defined on $U \supseteq E$. It is clear that 
	\begin{equation}
	|S'_{n,j}(z)|=\dfrac{1}{|(T^{n})'(x)|}\leq r,
		\tag{3.4}\label{eq:main3.4}
	\end{equation}
	by the inverse function theorem.   Then
	\begin{align*}
		\textup{dist}\left( x, J\right)
		& \leq |S_{n,j}(z)-S_{n,j}(y)|\\
		&\leq\sup_{u\in U'}|S'_{n,j}(u)||z-y|\\
		&\approx |S'_{n,j}(z)||z-y|~~\left( \text{by Lemma }\ref{lem:2}\right) \\
		&\leq r|z-y|~~\left( \text{by } (\ref{eq:main3.4})\right)  .
	\end{align*}
Since $ |z-y| \lesssim  1$, we know $ x $ is in the $ \lesssim r $ neighbourhood of $ J $.
\end{proof}

Next we aim to bound the number of pre-images close to the Julia set in terms of an appropriate covering number of the Julia set.  This will rely on Lemma \ref{lem4.1} above.

\begin{lemma}\label{lem4.3}
Let  $ z\in E\setminus J$. Then, for integers $ n \geq 0 $,
  	$$\#\left\lbrace \left( k,x\right) :k\in \mathbb{N} , x\in T^{-k}(z), 2^{n}\leq|(T^{k})'(x)|\leq2^{n+1} \right\rbrace \lesssim N_{2^{-n}}(J).$$
  	\end{lemma}
\begin{proof}
Let $ x $ and $ y $ be two distinct points in $ T^{-k}(z) $ for some $ k $ satisfying 
	
	\begin{equation}
		2^{n}\leq|(T^{k})'(x)|\leq2^{n+1} \quad \text{and} \quad 2^{n}\leq|(T^{k})'(y)|\leq2^{n+1}.
		\tag{3.5}\label{eq:main3.5}
	\end{equation}
We first prove that $ |x-y|\gtrsim \frac{1}{2^{n}}$.

Since $ z \in E \subseteq U $, we may choose $\kappa>0 $ such that $\overline{B(z,3\kappa)} \subseteq U$. In particular,    the inverse branch $ S_{k,j} $ is univalent on $ B(z,\kappa) $ for $ j=1,\dots,d^{k} $. Without loss of generality, assume that $ x=S_{k,1}(z) $ and $ y=S_{k,2}(z) .$ By Lemma \ref{pros4.4},
	$$4\textup{dist}\left( x, \partial S_{k,1}(B(z,\kappa) ) \right)\geq |S'_{k,1}(z)|\textup{dist}\left( z, \partial B(z,\kappa) \right)=|S'_{k,1}(z)|\kappa$$ 
	and $$4\textup{dist}\left( y, \partial S_{k,2}(B(z,\kappa) ) \right)\geq |S'_{k,2}(z)|\textup{dist}\left( z, \partial B(z,\kappa) \right)=|S'_{k,2}(z)|\kappa.$$
	Since $ S_{k,1}(B(z,\kappa) )\cap S_{k,2}(B(z,\kappa) )=\emptyset$ (otherwise, by analytic continuation, $ S_{k,1}(z) =S_{k,2}(z)$, which contradicts the definition of $S_{k,j}  $), it follows that
	\begin{align*}
		|x-y|&\geq \textup{dist}\left( x, \partial S_{k,1}(B(z,\kappa) ) \right)+\textup{dist}\left( y, \partial S_{k,2}(B(z,\kappa) ) \right)\\ 
		&\geq \dfrac{\kappa}{4}\dfrac{1}{2^{n}}~\left(\textup{by }\left( \ref{eq:main3.5}\right) \textup{~and the inverse function theorem} \right) \\
		&\gtrsim \dfrac{1}{2^{n}},
	\end{align*}
	as required. Using this separation estimate  and  that Lemma \ref{lem4.1} ensures the preimage of $z$ under $T^k$ is in the $\lesssim 2^{-n}$ neighbourhood of the Julia set, it follows that 
	\begin{equation}
		\#T^{-k}\left( z\right)\lesssim N_{2^{-n}}(J) .
		\tag{3.6}\label{eq:main3.6}
	\end{equation}
Next we consider the case when $x,y$ are not in a common preimage $T^{-k}(z)$.  First, choose $\rho>0$ sufficiently small such that
\[
T^k(B(z,\rho)) \cap B(z,\rho) = \emptyset
\]
for all $k \geq 1$.  To see why we can choose such a $\rho$, we use the classification of Fatou components   \cite[Theorem 3.2]{mcmullen} which says that the orbit $T^k(z)$ must accumulate on  an attracting cycle, a parabolic cycle, or else $z$ must belong to a  Siegel disk or a Herman ring.  By \cite[Corollary 3.7]{mcmullen} $P(T)$ contains   the boundary of every Siegel disk or Herman ring and therefore since $U \cap J \neq \emptyset$ and $z \in E \subseteq U$, $z$ cannot belong to a Herman ring or Siegel disk because $U$ cannot cross the (connected) boundary to intersect $J$, recalling that $U$ is connected. Moreover, $z$ itself cannot be in an attracting or parabolic cycle because $z \in E \subseteq U$ and $U \cap P(T) = \emptyset$ and, again by \cite[Corollary 3.7]{mcmullen}, $P(T)$  contains all parabolic and attracting cycles.  Now, choose $\rho_0>0$ small enough to ensure that $B(z,2\rho_0) \subseteq F$.  It follows that,  as $T^k(z)$ approaches a parabolic or attracting cycle,  the diameter of $T^k(B(z,\rho_0))$ must go to zero.  Therefore, only at most finitely many of the sets  $\{T^k(B(z,\rho_0))\}_{k \geq 1}$ intersect $B(z,\rho_0)$ and we are free to choose $\rho<\rho_0$ with the desired property. 

	Now, suppose that $ x\in T^{-m_{1}}(z) $ and $ y\in T^{-m_{2}}(z) $ where $ m_{1}\neq m_{2} $, and $ x $ and $ y $ satisfy (\ref{eq:main3.5}), where $ k $ in (\ref{eq:main3.5}) should be replaced by $m_{1}  $ and $ m_{2} $, respectively. Without loss of generality, assume that $m_{1}>m_{2} $.  We aim to show that $ |x-y|\gtrsim \frac{1}{2^{n}} $.  With the above argument in the $m_1=m_2$ case in mind,   it suffices to show that 
\[
T^{-m_{1}}(B(z,\rho))  \cap  T^{-m_{2}}(B(z,\rho)) = \emptyset.
\]
Suppose not and let $A = T^{-m_{1}}(B(z,\rho))  \cap  T^{-m_{2}}(B(z,\rho)) \neq \emptyset$.  Then
\[
T^{m_2}(A) \subseteq B(z,\rho)
\]
and 
\[
T^{m_1}(A)  = T^{m_1-m_2}(T^{m_2}(A))\subseteq B(z,\rho)
\]
but this contradicts the choice of $\rho$ and we may conclude that $ |x-y|\gtrsim \frac{1}{2^{n}} $.  The implicit constant here depends on $\rho$, but $\rho$ is a fixed constant independent of $n$. This separation  estimate, again combined with  with  Lemma \ref{lem4.1}, gives
$$\#\left\lbrace \left( k,x\right) :k\in \mathbb{N} , x\in T^{-k}(z), 2^{n}\leq|(T^{k})'(x)|\leq2^{n+1} \right\rbrace \lesssim N_{2^{-n}}(J)$$
completing the proof.
\end{proof}
Using Lemma  \ref{lem4.3}, we obtain the following result which establishes that the upper box dimension of the Julia set is at least the exponent of convergence of a certain series, somewhat analogous to the Poincar\'e exponent and Poincar\'e series from the theory of Kleinian groups.   This result, which  can be formulated so as to have nothing to do with $E$,  may be interesting in its own right. 
\begin{proposition}\label{pros4.3}
	 Let  $ z\in E \setminus J $   and let 
	$$h_{z}=\inf\left\lbrace s: \sum_{k=1 }^{\infty}\sum_{x\in T^{-k}(z)}|(T^{k})'(x)|^{-s}<\infty\right\rbrace .$$
	Then $ h_{z}\leq\overline{\dim}_{\mathrm{B}}J. $  This result can be made independent of $E$ by simply considering $E= \{z\}$ satisfying the assumptions of Theorem \ref{thm:A}.
\end{proposition}
\begin{proof}
Let $ s>t>\overline{\dim}_{\mathrm{B}}J $. First note that 
\[
\#\{k : |(T^{k})'(x)|\leq 1 \text{ for some $x\in T^{-k}(z)$} \} \lesssim 1
\]
by Lemma \ref{bimpliesa} and the fact that $J_T$ has empty interior. (Indeed, a ball must shrink if it is to be contained in a small neighbourhood of $J$.)  Recall that we assume throughout that $J_T \neq \mathbb{C}_\infty$ and it is known that in this case  $J_T$ has empty interior \cite[Theorem 4.2.3]{ref1}. Moreover,  since $ z \in E$, $|(T^{k})'(x)| > 0$ for all $k$ and $x\in T^{-k}(z)$.  In particular, 
\[
\sum_{k=1}^{\infty}\sum_{\substack{x:~x\in T^{-k}(z) \\  |(T^{k})'(x)|\leq 1}}|(T^{k})'(x)|^{-s} \lesssim 1.
\]
Therefore
\begin{align*}
	\sum_{k=1 }^{\infty}\sum_{x\in T^{-k}(z)}|(T^{k})'(x)|^{-s}
	&\lesssim 1 \ + \ \sum_{n=0}^{\infty} \sum_{k=1}^{\infty}\sum_{\substack{x:~x\in T^{-k}(z) \\ 2^{n}\leq |(T^{k})'(x)|\leq 2^{n+1}}}|(T^{k})'(x)|^{-s}\\
	&\approx \sum_{n=0}^{\infty} 2^{-ns} \sum_{k=1}^{\infty}\sum_{\substack{x:~x\in T^{-k}(z) \\ 2^{n}\leq |(T^{k})'(x)|\leq 2^{n+1}}}1\\
	&\lesssim\sum_{n=0}^{\infty} 2^{-ns}N_{2^{-n}}(J)~~\left( \textup{by Lemma } \ref{lem4.3}\right) \\
	&\lesssim\sum_{n=0}^{\infty} 2^{-ns}2^{nt}~~~\left( \textup{using }t>\overline{\dim}_{\mathrm{B}}J \right)\\
	&=\sum_{n=0}^{\infty} 2^{-n(s-t)}<\infty.
\end{align*}	
It follows that $ h_{z}\leq s$ and therefore $ h_{z}\leq\overline{\dim}_{\mathrm{B}}J $, as required.
\end{proof}

\subsection{Pulling it together: proof of Theorem \ref{thm:A}.}\label{proof:1}

	First, since $ E\subseteq O_{T}(E)$ and $J\subseteq \overline{O_{T}(E)} $, by the monotonicity and stability under closure of the box dimension, it follows that 
	$$\overline{\dim}_{\mathrm{B}}O_{T}(E) \geq\max\left\lbrace\overline{\dim}_{\mathrm{B}}E, ~\overline{\dim}_{\mathrm{B}}J \right\rbrace ,$$
	and it remains to prove the reverse inequality.   Let $ s>\max\left\lbrace\overline{\dim}_{\mathrm{B}}E, ~\overline{\dim}_{\mathrm{B}}J \right\rbrace, $ and fix $ \delta \in \left( 0,1\right).$ Then
	\begin{align*}
	N_{\delta}\left( \bigcup_{n=0}^{\infty}T^{-n}(E)\right) &= N_{\delta}\left( \bigcup_{n=0}^{\infty}\bigcup_{j=1}^{d^n} S_{n,j}(E)\right)\\
	&\leq N_{\delta}\left( \bigcup_{n,j:~S_{n,j}(E) \subseteq J_{\delta}}S_{n,j}(E)\right)+ N_{\delta}\left( \bigcup_{n,j:~S_{n,j}(E)\nsubseteq J_{\delta}}S_{n,j}(E)\right).
	\end{align*}
	Since 
	$$\bigcup_{n,j:~S_{n,j}(E) \subseteq J_{\delta}}S_{n,j}(E)\subseteq J_{\delta},$$
	we know
	$$N_{\delta}\left( \bigcup_{n,j:~S_{n,j}(E) \subseteq J_{\delta}}S_{n,j}(E)\right)\lesssim \delta^{-s} ,$$
	and so it remains to estimate the other term.
	
	We may assume without loss of generality that $E \nsubseteq J$ and therefore choose  $ z\in  E \setminus J $. Write $z_{n,j} = S_{n,j}(z) \in T^{-n}(z)$ for integers $n \geq 0$ and $j=1, \dots, d^n$.  Note that 
	\begin{equation}
		T^{n}( z_{n,j}) = z \in E\setminus J  \quad \text{and} \quad |(T^{n})'( z_{n,j})| \approx \frac{1}{|S'_{n,j}(E)|}.
		\tag{3.7}\label{eq:main3.7}
	\end{equation}
	Then
	\begin{align*}
	N_{\delta}\left( \bigcup_{n,j:~S_{n,j}(E)\nsubseteq J_{\delta}}S_{n,j}(E)\right) &\leq 		\sum_{n,j:~S_{n,j}(E)\nsubseteq J_{\delta}}N_{\delta}\left(S_{n,j}(E) \right)\\
		&= \sum_{n=0}^\infty \sum_{j:S_{n,j}(E)\nsubseteq J_{\delta}} N_{\delta}\left(S_{n,j}(E)\right)\\
		& \lesssim \sum_{n=0}^\infty \sum_{j:S_{n,j}(E)\nsubseteq J_{\delta}}  N_{\frac{\delta}{ |S'_{n,j} ( E) | }}(  E) ~~\left( \text{by Lemma }\ref{pros4.1}\right) \\
		&\lesssim\sum_{n=0}^\infty \sum_{j:S_{n,j}(E)\nsubseteq J_{\delta}}  \left(\frac{\delta}{ |S'_{n,j} ( E) | } \right) ^{-s}~~\left( \text{using }s>\overline{\dim}_{\mathrm{B}}E \text{ and Lemma } \ref{pros4.2}\right) \\
		&	\approx \delta^{-s} \sum_{n=0}^\infty \sum_{j:S_{n,j}(E)\nsubseteq J_{\delta}}  \left( |( T^{n}) '( z_{n,j}) |\right)^{-s}~~\left( \text{by }\eqref{eq:main3.7}\right)  \\ 
		&\leq \delta^{-s}\sum_{n=0}^{\infty}\sum_{j=1}^{d^{n}} \left( |( T^{n}) '( z_{n,j}) |\right)^{-s}\\
		&\lesssim\delta^{-s} ~~\left( \text{using } s>\overline{\dim}_{\mathrm{B}}J\text{ and Proposition }\ref{pros4.3} \right). 
	\end{align*}
	Consequently,
	$$N_{\delta}\left( \bigcup_{n=0}^{\infty}T^{-n}(E)\right)\lesssim \delta^{-s}$$
	and this  implies
	$$\overline{\dim}_{\mathrm{B}}O_{T}(E) \leq\max\left\lbrace\overline{\dim}_{\mathrm{B}}E, \overline{\dim}_{\mathrm{B}}J \right\rbrace ,$$
	completing the proof.

\subsection{Some counterexamples: proof of Proposition \ref{thm:B}}\label{proof:3}

	\textbf{Proof of (i)}. 
	Let $E= \left[ 0,1\right]$ and $T$ be defined by $T( z) =z^{2}$. In this case, $ 0$ is a critical point, an attracting fixed point, and an exceptional point of $ T$, and the Julia set of $ T $ is the unit circle $ S^{1} .$ For any point $ w \in \left\lbrace z:|z|<1\right\rbrace  $, 0 is the accumulation point of the sequence $ \left\lbrace T^{n}(w):n\geq1\right\rbrace . $ Therefore,   assumption (A) from Theorem \ref{thm:A} fails to hold and also the conclusion of Lemma \ref{bimpliesa} does not hold. 	We shall prove that the orbital set of $ E $ is dense in the unit disk $ \mathbb{D}^{2} $ from which is follows that
	$$\overline{\text{dim}}_{\text{B}}O_{T}(E)= \dim_{\text{H}}\overline{O_{T}(E)}=2.$$ However, 
	$$\overline{\dim}_{\mathrm{B}}E=\dim_{\mathrm{H}}J =\overline{\dim}_{\mathrm{B}}J = 1 $$
	and so the conclusion of Theorem \ref{thm:A} is false. 	Let $$ A=\left\lbrace z:z=re^{2\pi i\frac{q}{2^{m}}}, ~q,~m\in \mathbb{N},~0\leq r \leq 1\right\rbrace  .$$ It is clear that $ A $ is dense in the unit disk $ \mathbb{D}^{2} $. For any point $ z=re^{2\pi i\frac{q}{2^{m}}} \in A$, 
	$$ T^{m}( z) =r^{2^{m}}\in E ,$$
	which implies $z \in O_{T}(E)   .$ Therefore, $ A\subseteq O_{T}(E) $, and so $ O_{T}(E) $ is dense in the unit disk, as required.\\
	\textbf{Proof of (ii).} Let $ E=\left\lbrace  6+\frac{1}{n} \right\rbrace _{n\in \mathbb{N}} \cup \left\lbrace  6\right\rbrace $ and $T$ be defined by $ T(z)=z^{2}+6 $.  It is clear that 6 is a critical value of $ T $ and $ \infty $ is the unique attracting fixed point, and for any $ w\in F $, the sequence $ \left\lbrace T^{n}(w):n\geq1 \right\rbrace  $ tends to $ \infty $. Therefore,   we may indeed apply Lemma \ref{lem:1} to conclude that the pre-images of $E$ do accumulate on $J$.  However, assumption (A) from  Theorem \ref{thm:A}   does not hold.
	 It is easy to see, e.g. recall \cite[Example 2.7]{ref2}, that
	$$\overline{\dim}_{\mathrm{B}}E=\dfrac{1}{2}.$$
	Moreover,  for the rational map $T_c$ defined by $ T_{c} =z^{2}+c$ for  $ c\in\mathbb{C} $, 
	$$\overline{\dim}_{\mathrm{B}}J_{T_{c}}=\dim_{\mathrm{B}}J_{T_{c}} \leq \dfrac{2\log 2}{\log (4(|c|-\sqrt{2|c|}))}$$
	by \cite[Theorem 14.15]{ref2} and so we have 
	\[
	 \overline{\dim}_{\mathrm{B}}J_{T_{6}}<\frac{2}{3}
	 \]
	  in the case $c=6$, as required.
	
	Turning our attention to the orbital set, observe that $ T^{-1}(E)=\left\lbrace  \pm\sqrt{\frac{1}{n}} \right\rbrace _{n\in \mathbb{N}} \cup \left\lbrace  0\right\rbrace.$ In particular, an easy calculation gives 
 $$\overline{\text{dim}}_{\text{B}}T^{-1}(E)=\dfrac{2}{3}.$$
	Observe that the assumptions from Theorem \ref{thm:A} are satisfied for $T$ with $E$ replaced by $T^{-1}(E)$ and so  it follows from Theorem \ref{thm:A} (with $E$ replaced by $T^{-1}(E)$)   that 
	$$\overline{\text{dim}}_{\text{B}}O_{T}(E) =\dfrac{2}{3}$$
	where we use that $$O_T(E) = O_T(T^{-1}(E)) \cup E.$$
	 In particular, the conclusion of Theorem \ref{thm:A} fails for $O_{T}(E)$.\\
	 \textbf{Proof of (iii).} This example is similar to the previous one and so we only sketch it.  Let $\varepsilon>0$ and $p>0$ and let $T$ be defined by $T(z) = z^N + c$ where $p>0$  and $c>0$ are large and  depend on $\varepsilon$ and  $N \geq 1$ is a large integer depending on $p$.     Let $ E=\left\lbrace  c+\frac{1}{n^p} + \frac{i}{m^p} \right\rbrace _{n,m\in \mathbb{N}} \cup \left\lbrace  c\right\rbrace $.  We choose $p$ sufficiently large to ensure 
	 \[
	 \overline{\dim}_{\mathrm{B}} E = \frac{2}{1+p} < \varepsilon
	 \]
	 and $c$ is chosen sufficiently large to ensure that
	  \[
	 \overline{\dim}_{\mathrm{B}} J   < \varepsilon.
	 \]
	 We can do this by following the approach of \cite[Theorem 14.15]{ref2} where, for large enough $c$, $J$ is realised as an attractor of a conformal IFS. Then we see that
	 \[
	  T^{-1}(E)\supseteq \left\lbrace  \sqrt[N]{\frac{1}{n^{p}} + \frac{i}{m^{p}}} \right\rbrace _{n,m\in \mathbb{N}}  \cup \left\lbrace  0\right\rbrace
	  \]
	  and an easy calculation gives
	  \[
	  \overline{\dim}_{\mathrm{B}}  O_T(E) \geq \frac{2}{1+p/N} > 2- \varepsilon
	  \]
	  for sufficiently large $N=N(p)$.  This completes the proof. \\
\textbf{Proof of (iv)}. We give the proof in the Siegel disk case, the Herman ring case being similar. Recall that  $T$ admits a Siegel disk at $z_0 \in \mathbb{C}$ if there exists a  neighbourhood   $V$ of $0$ and a  homeomorphism $\phi $ mapping $V$ to $B(0,1)$ such that 
\[
\phi T^p \phi^{-1} (z) = e^{2\pi i \alpha} z
\]
for some period $p \geq 1$ and  some irrational $\alpha$.  That is, the dynamics of $T^p$ are conjugate to an irrational rotation near 0.  Choosing $E = \phi^{-1}([0,1/2])$ (or $E = \phi^{-1}([1/3,1/2])$ if one wants $E$ to avoid the fixed point) we get
\[
\overline{O_T(E)} \supseteq \phi^{-1} \{ e^{2\pi i \theta} z : \theta \in [0,1], z \in E\}
\]
has non-empty interior and therefore 
$$\overline{\text{dim}}_{\text{B}}O_{T}(E)=2$$
as required.  Moreover, by \cite[Theorem 4.1]{mcmullendisks}, $\overline{\text{dim}}_{\text{B}} J_T <2$ when $T$ is given by $T(z) = e^{2 \pi i \alpha} z + z^2$ with $\alpha$ an irrational number of bounded type.  In particular, such $T$ admit a Siegel disk at 0. Finally,  such $E$ do not intersect  $P(T)$ and therefore assumption (A) would be satisfied if it was weakened to allow $U \cap J = \emptyset$.

\subsection{Hausdorff dimension: proof of Proposition \ref{thm: C}.}\label{proof:2}

	By \cite[Theorem 2.3.1]{ref1}, any rational map satisfies the Lipschitz condition on the Riemann sphere with respect to the spherical metric, that is, 
	$$\sigma_{0}\left(Tz,Tw\right) \leq M\sigma_{0}\left(z,w\right),$$
	where $\sigma_{0}  $ is the spherical metric and $ M>0 $ is the supremum of the ratio
	$$\dfrac{|T'( z) |\left( 1+|z|^{2}\right) }{1+|T( z) |^{2}}.$$
	
	If $ E $ does not contain any critical values of the map $ T $, then by the inverse function theorem and the Lipschitz condition described above, any inverse branch $ S_{j} $ of $ T^{-1} $ defined on $ E $ is holomorphic with uniformly bounded derivative, that is, $  |S'_{j}( z)|  $ is uniformly bounded for all $ z\in E. $ In this case, it is easy to see
	$$\dim_{\mathrm{H}}(T^{-1}(E))=\dim_{\mathrm{H}}E.$$
	
	On the other hand, suppose   $ E $ does  contain critical values of $ T $. Note that the set of critical values  $T(\mathrm{Crit}( T))$ is finite. Fix $ n\in \mathbb{N} $, and let $B_{n}= T(\mathrm{Crit}( T))_{1/n}$ be the $1/n$-neighbourhood of the set of critical values and assume $n$ is large enough such that the complement of $B_n$ is simply connected.  We can decompose $ E\setminus B_{n}$ into two parts $ E_{n,1} $ and $ E_{n,2} $ such that $ E_{n,j} $ is contained in a simply connected domain $ U_{n,j}$ $(j=1,2) $ and each $ U_{n,j} $ does not contain critical values of $ T $. It follows that every branch of $ T^{-1} $ is univalent and analytic on $ U_{n,j} $. 	Then   $$ \dim_{\mathrm{H}}(T^{-1}(E_{n,j}))=\dim_{\mathrm{H}}E_{n,j} ~~~(j=1,2).$$
	It follows that
	\begin{align*}
		\dim_{\mathrm{H}}(T^{-1}(E\setminus B_{n}))&=\max\left\lbrace \dim_{\mathrm{H}}(T^{-1}(E_{n,1})),\dim_{\mathrm{H}}(T^{-1}(E_{n,2}))\right\rbrace\\
		&=\max\left\lbrace \dim_{\mathrm{H}}E_{n,1},\dim_{\mathrm{H}}E_{n,2}\right\rbrace\\
		&=\dim_{\mathrm{H}}(E\setminus B_{n}).
	\end{align*}
Thus, 
	\begin{align*}
		\dim_{\mathrm{H}}T^{-1}( E)
		&=\dim_{\mathrm{H}}T^{-1}\big( E\setminus T(\mathrm{Crit}( T)) \big)\\
		&=\sup_{n}\left\lbrace \dim_{\mathrm{H}}T^{-1}( E\setminus B_{n})\right\rbrace\\
		&	=\sup_{n}\left\lbrace \dim_{\mathrm{H}}( E\setminus B_{n})\right\rbrace\\
		&	=\dim_{\mathrm{H}}\big( E\setminus T( \mathrm{Crit}( T))\big)\\
		&	=\dim_{\mathrm{H}}E.
	\end{align*} 
	By the countable stability of Hausdorff dimension, we get
	$$\dim_{\mathrm{H}}O_{T}(E)=\dim_{\mathrm{H}}E,$$
	completing the proof.


\begin{thebibliography}{100}
	
	\bibitem[A79]{ref3} L. V. Ahlfors, \textit{Complex Analysis}, McGraw-Hill Book Company, 3rd. ed., (1979).
	
	\bibitem[BFM19]{ref18}S. Baker, J. M. Fraser and  A. Máthé, Inhomogeneous self-similar sets with overlaps, \textit{Ergodic Th. Dynam. Syst.}, \textbf{39}, (2019), 1-18.
	
	\bibitem[BKN23+]{kaenmaki1} B. B\'ar\'any, A. K\"aenm\"aki and P. Nissinen, Covering number on inhomogeneous graph-directed self-similar sets, preprint available at:
https://arxiv.org/abs/2307.16263
	
	\bibitem[BD85]{ref13}M. F. Barnsley and S. Demko, Iterated function systems and the global construction of fractals, \textit{Proc. R. Soc. Lond. Ser. A}, \textbf{399}, (1985), 243-275.
	
	\bibitem[BF23]{ref5}T. Bartlett and  J. M. Fraser, Dimensions of Kleinian orbital set, \textit{J. Fractal Geom.}, \textbf{10}, (2023), 267-278.
	
	\bibitem[B90]{ref1}A. F. Beardon, \textit{Iteration of Rational Functions}, Graduate Texts in Mathematics,  \textbf{132}, Springer-Verlag, New York, (1990).
	
	\bibitem[BF20]{ref17}S. A. Burrell and J. M. Fraser, The dimensions of inhomogeneous self-affine sets, \textit{Ann. Acad. Sci. Fenn. Math.}, \textbf{45}, (2020), 313-324.
	
	\bibitem[C95]{ref16}J. B. Conway, \textit{Functions of One Complex Variable $ \uppercase\expandafter{\romannumeral 2} $ }, Graduate Texts in Mathematics, \textbf{159}, Springer-Verlag, New York, (1995).
	
	
	\bibitem[F14]{ref2}K. J. Falconer, \textit{Fractal Geometry: Mathematical Foundations and Applications}, John Wiley $\& $ Sons, Hoboken, NJ, 3rd. ed., (2014).
	
	\bibitem[F12]{ref6}J. M. Fraser, Inhomogeneous self-similar sets and box dimensions, \textit{Studia Math.}, \textbf{213}, (2012), 133-156.
	
	\bibitem[F16]{ref19}J. M. Fraser, Inhomogeneous self-affine carpets, \textit{Indiana Univ. Math. J. }, \textbf{65}, (2016), 1547-1566.
	
	\bibitem[F25]{ref20}J. M. Fraser, Inhomogeneous attractors and box dimension,   \textit{Contemp. Math.}, {\bf 825},  (2025), 67--85.
	
	
\bibitem[KL17]{kaenmaki2}	A. K\"aenm\"aki and J. Lehrb\"ack, Measures with predetermined regularity and inhomogeneous self-similar sets, {\em  Ark. Mat.}, {\bf 55}, (2017), 165-184.

	
	\bibitem[L86]{ref10}M. Y. Lyubich, The dynamics of rational transforms: the topological picture, \textit{Uspekhi Mat. Nauk,} \textbf{41} (1986), 35-95; \textit{Russian Math. Surveys,} \textbf{41} (1986), 43-117.
	
	\bibitem[M94]{mcmullen}	C. T. McMullen, \emph{Complex dynamics and renormalization}, Annals of Mathematics Studies, {\bf 135}, Princeton University Press, Princeton, NJ, (1994).
	
\bibitem[M98]{mcmullendisks}	C. T. McMullen,
Self-similarity of Siegel disks and Hausdorff dimension of Julia sets, \emph{Acta Math.}, {\bf 180},  (1998),   247-292.
	
	

\bibitem[OS07]{olsen} L. Olsen and N. Snigireva,  $L^q$ spectra and R\'enyi dimensions of in-homogeneous self-similar measures, \emph{Nonlinearity}, {\bf 20}, (2007), 151-175.

\bibitem[OR25+]{rutar}
V. Orgoványi and A. Rutar, Two-scale branching functions and inhomogeneous attractors, preprint available at:
https://arxiv.org/abs/2510.07013
	
	
	
   
\end{thebibliography}
\end{document}